\documentclass[a4paper,11pt]{article}
\usepackage{hyperref}
\usepackage{changes}
\hypersetup{
    colorlinks=true,
    linktoc=all,
    linkcolor=black,     
    citecolor=blue,     
}
\usepackage{tikz}
\usetikzlibrary{positioning}
\usepackage{enumitem} 
\usepackage{cite}
\usepackage{setspace}
\usepackage{amssymb,amsthm}
\usepackage{latexsym}
\usepackage{amsfonts}
\usepackage{amsmath}
\allowdisplaybreaks
\newtheorem{theorem}{Theorem}[section]
\newtheorem{lemma}[theorem]{Lemma}
\newtheorem{claim}[theorem]{Claim}

\begin{document}
\begin{spacing}{1}

\title{Ramsey goodness of complete multipartite graphs \\with one large part}
\date{}

\author{
Shaonan Mi,\footnote{College of Mathematical Sciences, Harbin Engineering University, Harbin 150001, China. Email: {\tt mishaonan@hrbeu.edu.cn}. Supported in part by Natural Science Foundation of Heilongjiang Province of China (No.\ YQ2023A007).} \;\;\;\; Ye Wang\footnote{Corresponding author. College of Mathematical Sciences, Harbin Engineering University, Harbin 150001, China. Email: {\tt ywang@hrbeu.edu.cn}. Supported in part by NSFC (No.\ 12471323).}
}
\maketitle

\begin{abstract}
For graph $G$, a connected graph $H$ of order $n$ is $G$-good if $r(G,H)=(\chi(G)-1)(n-1)+s(G)$, 
where $\chi(G)$ is the chromatic number of $G$ and $s(G)$ is the minimum size of a color
class in a $\chi(G)$-coloring of $G$. 
Let $K_{\alpha_{1},\ldots ,\alpha_{p},n}$ be the complete $(p+1)$-partite graph with partite sets of sizes $\alpha_1,\ldots,\alpha_p,n$.
Burr, Faudree, Rousseau and Schelp (1983) showed that  $K_{\alpha_1,\ldots,\alpha_p,n}$ are $(K_2+mK_1)$-good for large $n$.
We determine graphs $G$ such that $K_{\alpha_{1},\ldots ,\alpha_{p},n}$ are $G$-good for large $n$.
The characterization depends on $\mathrm{snd}(\alpha_i)$, the smallest non-divisor of $\alpha_i$, where $1\le i\le p$.
\medskip

\end{abstract}

\section{Introduction}
\indent

Given graphs $G$ and $H$, the Ramsey number $r(G,H)$ is the smallest $N$ such that every red-blue edge coloring of $K_N$ contains either a red copy of $G$ or a blue copy of $H$.
Let $\chi(G)$ denote the chromatic number of $G$, and let $s(G)$ be the minimum size of a color class in a proper $\chi(G)$-coloring of $G$.
We denote by $v(G)$ and $e(G)$ the numbers of vertices and edges of $G$, respectively.

For vertex-disjoint graphs $G$ and $H$, we write $G \cup H$ for their union and $G+H$ for their join.
We denote by $nG$ the union of $n$ disjoint copies of $G$.
Let $K_{\alpha_{1},\ldots ,\alpha_{p},n}$ be the complete $(p+1)$-partite graph with partite sets of sizes $\alpha_1,\ldots,\alpha_p,n$.
In particular, we write  $K_{p}(\alpha)$ for  the complete $p$-partite graph $K_{\alpha,\ldots,\alpha}$.

Burr \cite{Bur81} showed that for every connected graph $H$ with $v(H)\ge s(G)$,
\begin{equation}{\label{bur}}
	r (G, H) \geq (\chi (G) - 1) (v(H) - 1) + s (G).
\end{equation}
Following Burr and Erd\H{o}s \cite{BE83}, a connected graph  $H$  is  $G$-good if the equality holds in \eqref{bur}.

A fundamental result of Chv\'atal \cite{Chv77} shows that trees are $K_{m}$-good,
which initiated systematic research on Ramsey goodness; see, for example, the survey of Conlon, Fox and Sudakov \cite{CFS15}.
Book graphs are bridges connecting small cliques and large complete multipartite graphs in Ramsey theory, see \cite{LR96A,LRZ01,CGMS}.
Burr, Faudree, Rousseau and Schelp \cite{BFRS83} showed that the complete multipartite  graphs  with one large part are book-good.

\begin{theorem}[Burr, Faudree, Rousseau and Schelp \cite{BFRS83}] \label{BFRS83}
Let $m\ge 1$, $p\ge 1$ and $\alpha_1,\ldots,\alpha_p\ge 1$ be integers.
Then $K_{\alpha_1,\ldots,\alpha_p,n}$ are $(K_2+mK_1)$-good for large $n$.
\end{theorem}

Burr and Faudree  \cite{BF93} characterized graphs $G$ such that large trees are $G$-good.
\begin{theorem}[Burr and Faudree \cite{BF93}] 
Let $G$ be a graph with $v(G)=m$, $\chi (G) = k + 1 \geq 2$ and $s(G) = 1$.
Then all large trees are $G$-good if and only if $G$ is a subgraph of $m K_2 + K_{k - 1}(m)$.
\end{theorem}

Liu and Li \cite{LL25} characterized graphs $G$ such that large generalized books are $G$-good.

\begin{theorem}[Liu and Li \cite{LL25}] \label{ll25}
Let $p \geq 2$ be an integer, and let $G$ and $H$ be graphs with $v(G)=m$, $\chi (G) = k + 1 \geq 2$ and $s(G) = 1$.
Then $K_{p} + nH$ are $G$-good for large $n$ if and only if $G$ is a subgraph of $m K_2 + K_{k - 1}(m)$.
\end{theorem}

These results naturally raise the question of whether similar Ramsey goodness phenomena persist for more general multipartite graphs.
In this paper, we characterize graphs $G$ for which $K_{\alpha_1,\ldots,\alpha_p,n}$ are $G$-good for large $n$,
which covers Theorem \ref{BFRS83} by Burr, Faudree, Rousseau and Schelp \cite{BFRS83}.
For a positive integer $\alpha$, let $\mathrm{snd}(\alpha)=\min\{k\in N^+:k\nmid\alpha\}$, the smallest non-divisor of $\alpha$.
This parameter determines the structure of the extremal graphs in our result.

\begin{theorem}\label{main}
For integers $p\ge 1$ and $\alpha_1,\ldots, \alpha_p\ge 1$ ,
if $G$ is a graph with $v(G)=m$, $\chi(G)=k+1\ge 2$ and $s(G)=1$,  then $K_{\alpha_1,\ldots,\alpha_p,n}$ are $G$-good for all large $n$ if and only if  $G$ is a subgraph of $mT + K_{k-1}(m)$ for every tree $T$ with  $v(T)=\min\{\mathrm{snd}(\alpha_1),\ldots,\allowbreak\mathrm{snd}(\alpha_p)\}$.
\end{theorem}

Let $t=\min\{\mathrm{snd}(\alpha_1),\ldots,\allowbreak\mathrm{snd}(\alpha_p)\}$.
Note that for graphs $G$ and $H$, $G$ is a subgraph of $H$ means that there is a graph embedding from $G$ to $H$.
That is, there is an injective mapping $\varphi: V(G)\to V(H)$ such that the adjacency is preserved:
for every edge $uv\in E(G)$, $\varphi(u)\varphi(v)\in E(H)$.
By $G$ being a subgraph of $mT + K_{k-1}(m)$ for every tree $T$ with $v(T)=t$,
we mean that there is a graph embedding $\varphi_T$ from $G$ to $ mT + K_{k-1}(m)$ for each $T$ with $v(T)=t$, where $\varphi_T$ may differ for different choices of $T$.
Let us look at some simple cases.
In the case $t=2$ for odd $\alpha_1$, the graph $G$ in Theorem \ref{main} is a subgraph of $mK_2+K_{k-1}(m)$.
When $t=3$, for example $\alpha_1=\dots=\alpha_p=2$, the graph $G$ in Theorem \ref{main} is a subgraph of $mK_{1,2}+K_{k-1}(m)$.
When $t=4$, for example $\alpha_1=\dots=\alpha_p=6$, the graph $G$ in Theorem \ref{main} is a subgraph of both $mK_{1,3}+K_{k-1}(m)$ and $mP_4+K_{k-1}(m)$, where $P_4$ is a path on four vertices.
\medskip

\noindent{\bf Remark.}
The condition that $G$ is a subgraph of $mT + K_{k-1}(m)$ for every tree $T$ with $v(T)=\min\{\mathrm{snd}(\alpha_1),\ldots,\allowbreak\mathrm{snd}(\alpha_p)\}$ in Theorem \ref{main} 
does not mean that $G$ is a subgraph of $m\big(\bigcap_{T\in\mathcal T}T\big)+K_{k-1}(m)$ where $\mathcal T=\left\{T:\ T\text{ is a tree and }v(T)=\min\{\operatorname{snd}(\alpha_1),\ldots,\operatorname{snd}(\alpha_p)\}\right\}$, and hence of $mK_{1,2}+K_{k-1}(m)$.
For example, $G=P_6+K_1$ with $V(G)=\{w_1,\ldots,w_6,y\}$ and $E(G)=\{w_iw_{i+1},yw_j\}$ for $1\le i\le 5$ and $1\le j\le 6$,
where $m=7$, $\chi(G)=3$ and $s(G)=1$.
For $\alpha_1=\dots=\alpha_p=60$, we have $\min\{\mathrm{snd}(\alpha_1),\ldots,\allowbreak\mathrm{snd}(\alpha_p)\}=7$.
It is easy to see that every tree is either a path or contains a $K_{1,3}$ as a subgraph.
If the tree on seven vertices is a path $u_1u_2\dots u_7$, then there is a graph embedding
$$\varphi_T:w_1,w_2,w_3,w_4,w_5,w_6,y\to u_1,u_2,u_3,u_4,u_5,u_6,v_1, \mbox{respectively,}$$
 see Figure \ref{fig1}(a) and (b1).
If the tree contains a $K_{1,3}$ as a subgraph with edges $u_1u_2$, $u_1u_3$ and $u_1u_4$,
then there is a graph embedding
$$\varphi_T:w_1,w_2,w_3,w_4,w_5,w_6,y\to u_2,v_1,u_3,v_2,u_4,v_3,u_1, \mbox{respectively,}$$
see Figure \ref{fig1}(a) and (b2).
So $G$ is a subgraph of $T + 7K_1$, and hence of $7T+7K_1$ for any tree $T$ on seven vertices.
However, $G=P_6+K_1$ is not a subgraph of $7K_{1,2}+7K_1$.

\begin{figure}[htbp]\label{fig1}
\centering
\begin{tikzpicture}[
scale=0.65,
transform shape,
solid/.style={circle,draw,fill=black,inner sep=1.3pt},
every node/.style={font=\Large}
]

\begin{scope}[shift={(0,0)}]
\node[solid] (w1) at (0,2.0) {};
\node[solid] (w2) at (0,1.2) {};
\node[solid] (w3) at (0,0.4) {};
\node[solid] (w4) at (0,-0.4) {};
\node[solid] (w5) at (0,-1.2) {};
\node[solid] (w6) at (0,-2.0) {};
\node[solid] (y) at (3,0) {};

\node[left=0.1cm] at (w1) {$w_1$};
\node[left=0.1cm] at (w2) {$w_2$};
\node[left=0.1cm] at (w3) {$w_3$};
\node[left=0.1cm] at (w4) {$w_4$};
\node[left=0.1cm] at (w5) {$w_5$};
\node[left=0.1cm] at (w6) {$w_6$};
\node[right=0.1cm] at (y) {$y$};

\draw[thick] (w1)--(w2)--(w3)--(w4)--(w5)--(w6);
\foreach \i in {1,...,6}{
  \draw[thick] (w\i)--(y);
}

\node[font=\Large] at (1.5,-4.5) {(a) $G=P_6+K_1$};
\end{scope}

\begin{scope}[shift={(8,0)}]
\draw (0,0) ellipse (1.5 and 3.5);
\node at (0,4) {$T$};
\draw (4,0) ellipse (1.5 and 3.5);
\node at (4,4) {$7K_1$};

\node[solid] (u1) at (0,2.4) {};
\node[solid] (u2) at (0,1.6) {};
\node[solid] (u3) at (0,0.8) {};
\node[solid] (u4) at (0,0) {};
\node[solid] (u5) at (0,-0.8) {};
\node[solid] (u6) at (0,-1.6) {};
\node[solid] (u7) at (0,-2.4) {};

\foreach \i in {1,...,7}{
  \node[left=0.1cm] at (u\i) {$u_\i$};
}

\draw[thick] (u1)--(u2)--(u3)--(u4)--(u5)--(u6)--(u7);

\node[solid] (v1) at (4,2.4) {};
\node[solid] (v2) at (4,1.6) {};
\node[solid] (v3) at (4,0.8) {};
\node[solid] (v4) at (4,0) {};
\node[solid] (v5) at (4,-0.8) {};
\node[solid] (v6) at (4,-1.6) {};
\node[solid] (v7) at (4,-2.4) {};

\foreach \i in {1,...,7}{
  \node[right=0.1cm] at (v\i) {$v_\i$};
}

\draw[dashed] (1.5,1) -- (2.6,1);
\draw[dashed] (1.5,0) -- (2.5,0);
\draw[dashed] (1.5,-1) -- (2.6,-1);

\node[font=\Large] at (2,-4.5) {(b1) $T=P_7$};
\end{scope}

\begin{scope}[shift={(18,0)}]
\draw (0,0) ellipse (1.5 and 3.5);
\node at (0,4) {$T$};
\draw (4,0) ellipse (1.5 and 3.5);
\node at (4,4) {$7K_1$};

\node[solid] (u1) at (-0.5,0) {};
\node[solid] (u2) at (0.5,1.2) {};
\node[solid] (u3) at (0.5,0) {};
\node[solid] (u4) at (0.5,-1.2) {};

\node[left] at (u1) {$u_1$};
\node[right] at (u2) {$u_2$};
\node[right] at (u3) {$u_3$};
\node[right] at (u4) {$u_4$};
\node at (0,-2) {$\dots$};

\draw[ thick] (u1)--(u2);
\draw[thick] (u1)--(u3);
\draw[thick] (u1)--(u4);

\node[solid] (v1) at (4,2.4) {};
\node[solid] (v2) at (4,1.6) {};
\node[solid] (v3) at (4,0.8) {};
\node[solid] (v4) at (4,0) {};
\node[solid] (v5) at (4,-0.8) {};
\node[solid] (v6) at (4,-1.6) {};
\node[solid] (v7) at (4,-2.4) {};

\foreach \i in {1,...,7}{
  \node[right=0.1cm] at (v\i) {$v_\i$};
}

\draw[dashed] (1.5,1) -- (2.6,1);
\draw[dashed] (1.5,0) -- (2.5,0);
\draw[dashed] (1.5,-1) -- (2.6,-1);

\node[font=\Large] at (2,-4.5) {(b2) $K_{1,3}\subseteq T$};
\end{scope}

\end{tikzpicture}
\caption{Two embeddings of $P_6+K_1$ into $T+7K_1$.}
\label{fig:clean}
\end{figure}
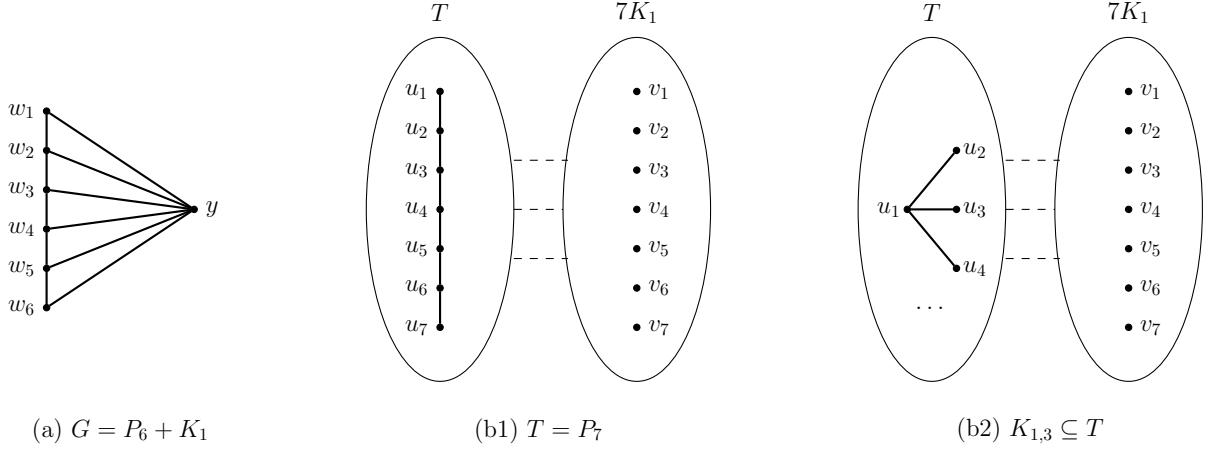
		
\section{Preliminaries}

In this section, we list and prove some preliminary results, which we need for the proofs.
To state these results, we first introduce some standard notation.
For graph $G$, let  $\Delta(G)$ be the maximum degree of  $G$.
For $v\in V(G)$, the neighborhood of $v$  in $G$  is denoted by  $N_G(v)$  or  $N(v)$  if no danger of confusion,
and  for $U \subseteq V(G)$, let $d_U(v)$ be the number of neighbors of $v$  in $U$.
Let $U$ and $V$ be two disjoint vertex sets in a graph.
The density $d(U,V)=e(U,V)/|U||V|$,
where $e(U,V)$ is the number of edges between $U$ and $V$.
For $ \epsilon>0$, $(U,V)$ is said to be $\epsilon$-regular if, for all sets $U'\subseteq U$,
$V'\subseteq V$ with $|U'|\geq \epsilon |U|$, $|V'|\geq \epsilon |V|$,
$|d(U',V')-d(U,V)|<\epsilon.$

For positive functions $f(n)$ and $g(n)$, we write $f = o(g)$  if  $f / g \to 0$ as $n\to \infty$,
$f = O(g)$  if  $f \leq c g$  for all large  $n$, where  $c > 0$  is a constant, $f = \Omega(g)$  if  $g = O(f)$,  and  $f = \Theta(g)$  if  $f = O(g)$  and  $g = O(f)$.

The well-known Erd\H{o}s-Simonovits stability theorem can be found in \cite{Erd66,Erd68,ES66,Sim68}, and we use the following form from \cite{LR96} for convenience.

\begin{lemma}{\label{wending}}
Let  $G$  be a given forbidden graph with  $\chi(G) = k + 1$. For every $\xi > 0$, there exist $\delta = \delta(\xi) > 0$  and $n_0 = n_0(\xi) > 0$
such that if $H$ is a graph of order $N > n_0$ and  $e(H) > \frac{k - 1}{2k} N^2 - \delta N^2$ that does not contain $G$, then $V(H)$  has a partition  $V_1, \ldots, V_k$  such that
\begin{itemize}
\item $\frac{N}{k} - \xi N < |V_i| < \frac{N}{k} + \xi N$  for each  $i=1,2,\ldots,k$;
\item all but at most  $\xi N^2$  pairs $\{u,v\}$ with $u\in V_i$, $v\in V_j$ $(i\neq j)$ belong to $E(H)$;
\item at most  $\xi N^2$ pairs $\{u,v\}$  with $u,v\in V_i$ belong to $E(H)$;
\item for any  $u\in V_i$,  $d_{V_i}(u) \le d_{V_j}(u)$  $(i\neq j)$  for each  $j=1,2,\ldots,k$.
\end{itemize}
\end{lemma}

The following version of Szemer\'edi's regularity lemma and some technical results are needed in the proofs.

\begin{lemma}[Szemer\'edi \cite{Sze78}, Koml\'os and Simonovits \cite{KS96}]{\label{zhengzeyinli}}
For every  $0<\epsilon<1$ and positive integer $\ell$, there exists  $M = M(\epsilon ,\ell) > 0$  such that every graph  $G$ with at least  $n$  vertices
has a partition  $V(G)=\cup_{i=0}^k V_i$, where  $\ell \leq k\leq M$, which satisfies
\begin{itemize}
\item $|V_{1}| = |V_{2}| = \dots = |V_{k}|$, and  $|V_{0}| < \epsilon n$;
\item all but  $\epsilon k^2$ pairs $(V_i, V_j)$ with $1\le i\neq j\le m$ are  $\epsilon$-regular.
\end{itemize}
\end{lemma}

\begin{lemma}[Erd\H{o}s \cite{Erd62}]{\label{jishuqian}}
For integer  $p\geq 2$ and graph  $G$ of order $m$,  any  $G$-free graph  $F$  of order  $n \geq r(G, K_{p})$  contains at least  $cn^{p}$  independent $p$-sets, where  $c = c(m,p) > 0$ is a constant.
\end{lemma}

\begin{lemma}[Koml\'os and Simonovits \cite{KS96}, Fact 1.4]{\label{l1}}
	Suppose  $0 < \epsilon < d \leq 1$  and  $(A, B)$  is an  $\epsilon$-regular pair with  $d(A, B) = d$. If  $Y \subseteq B$  and  $(d - \epsilon)^{r-1}|Y| > \epsilon|B|$  where  $r\ge 2$, then there are at most  $\epsilon r|A|^r$ $r$-sets  $R \subseteq A$  with
	$$
	\big| \big(\bigcap_ {u \in R} N (u)\big) \cap Y \big| \leq (d - \epsilon) ^ {r} | Y |.
	$$
\end{lemma}

We shall use Lemma \ref{l1} to obtain a similar result as Fact 1.4 in \cite{KS96} and Lemma 2 in \cite{NR04},
which estimates the number of copies of $K_{\alpha_1,\ldots,\alpha_p,1}$ under regularity conditions.

\begin{lemma}{\label{l2}}
Suppose $p\ge 2$, $\alpha=\sum_{i=1}^p \alpha_i$ with $\alpha_i\ge 1$,  $0 < \epsilon < d \leqslant 1$  and  $(d - \epsilon)^{\alpha - 2} > \epsilon$. For graph $H$, if $V(H)$ has a partition  $A,B_1,\dots,B_k$ with $|A| = |B_1| = \dots = |B_k|$  and the pair  $(A, B_i)$  is  $\epsilon$-regular with  $d(A, B_i) \geqslant d$ for every  $i \in [k]$. If  $q$  is the number of copies of  $K_{\alpha_1,\ldots,\alpha_p}$ in  $A$, then $H$ contains at least
	$$
	k | A | (q - \epsilon \alpha | A | ^ {\alpha}) (d - \epsilon) ^ {\alpha}
	$$
	$K_{\alpha_1,\ldots,\alpha_p,1}$ with the  $\alpha$  vertices of $K_{\alpha_1,\ldots,\alpha_p}$ in  $A$.
\end{lemma}

\begin{proof}
For every  $i \in [k]$, applying Lemma \ref{l1} to the pair  $(A, B_i)$  with  $Y = B_i$,
there are at most  $\epsilon \alpha |A|^{\alpha}$ $\alpha$-sets $R$ in  $A$  with
	$$
	\Big| \Big(\bigcap_ {u \in R} N (u)\Big) \cap B _ {i} \Big| \leqslant (d - \epsilon) ^ {\alpha} |A|,
	$$
	and therefore, at least  $(q - \epsilon \alpha |A|^{\alpha})$ $K_{\alpha_1,\ldots,\alpha_p}$ in $A$  satisfy
	$$
	\Big| \Big(\bigcap_ {u \in R} N (u)\Big) \cap B _ {i} \Big|  > (d - \epsilon) ^ {\alpha} |A|.
	$$
	Hence, $H$ contains at least  $k(d - \epsilon)^{\alpha}|A| (q - \epsilon \alpha |A|^{\alpha})$  $K_{\alpha_1,\ldots,\alpha_p,1}$ with the $\alpha$  vertices of $K_{\alpha_1,\ldots,\alpha_p}$ in  $A$  and one vertex in  $\bigcup_{i\in [k]}B_i$.
\end{proof}

The following is a special form of blow-up lemma \cite{KSS97}, in which each vertex of  $F$  is replaced by  $m$  additional vertices and each edge of  $F$  is replaced by an edge set of density  $d$.

\begin{lemma}[Koml\'os and Simonovits \cite{KS96}]{\label{l3}}
	Let  $d$  and  $\epsilon$  be real numbers with  $0< \epsilon <d<1$  and integer  $m \geq 1$. For graph  $F$, construct a graph  $G$  from  $F$  by replacing every vertex of  $F$  by  $m$ vertices, and replacing the edges of $F$ with $\epsilon$-regular pairs of density at least $d$.
   For any subgraph $H$ of $F$ with $\Delta=\Delta(H)>0$, if $\epsilon\le (d-\epsilon)^{\Delta}/(\Delta+2)$, then $H$ is a subgraph of $G$.
\end{lemma}

The following result is the well-known Erd\H{o}s-Stone theorem \cite{ES46}, in which the bound of $t$ comes from Bollob\'as and Erd\H{o}s \cite{BE73}.

\begin{lemma}[Erd\H{o}s and Stone \cite{ES46}, Bollob\'as and Erd\H{o}s \cite{BE73}]{\label{l4}}
	Let  $\epsilon > 0$  be a real number and let  $k \geq 2$  be an integer. Then any graph  $F$  on  $n$  vertices and at least
	$$
	\left(\frac {k - 2}{k - 1} + \epsilon\right) \binom {n} {2}
	$$
	edges contains a  $K_{k}(t)$, where  $t = \Omega (\log n)$  as  $n\to \infty$.
\end{lemma}

\section{Sufficiency: Extremal colorings and forbidden embeddings}\label{mainproof}

In this section, we prove the  sufficiency  of Theorem \ref{main} as follows.

\begin{proof}
Let $T$ be any tree on $t$ vertices, where $t=\min\{\mathrm{snd}(\alpha_1),\ldots,\allowbreak\mathrm{snd}(\alpha_p)\}$.
As $K_{\alpha_1,\ldots,\alpha_p,n}$ are $G$-good for all large $n$,
$$r(G,K_{\alpha_1,\ldots,\alpha_p,n})=k(\alpha+n-1)+1,$$
where $\alpha=\sum_{i=1}^{p}\alpha_i$.
Let  $n$ be integers such that $t\mid (\alpha+n)$.
Then we color the edges of $K_{k(\alpha+n-1)+1}$ by red and blue such that the blue graph is
	$$(k-1)K_{\alpha+n-1}\cup (K_{\alpha+n} \setminus \frac{\alpha+n}{t}T).$$
	Note that $(k-1)K_{\alpha+n-1}$ contains no $K_{\alpha_1,\ldots,\alpha_p,n}$.
    Consider $K_{\alpha+n} \setminus \frac{\alpha+n}{t}T$, and we divide the vertices into sets $V_1,V_2,\ldots,V_{(\alpha+n)/t}$ such that each $V_i$ is the set of vertices of the deleted $T$ for $1\le i\le (\alpha+n)/t$.
   Suppose that $K_{\alpha_1,\ldots,\alpha_p,n}$ is a subgraph of $ K_{\alpha+n} \setminus \frac{\alpha+n}{t}T$.
 For any vertex $v\in V_i$,  if $v$ is in some partite set of size $\alpha_j$ in $K_{\alpha_1,\ldots,\alpha_p,n}$,
 the neighbor of $v$ in $T$, denoted by $u$, must be in the same partite set of size $\alpha_j$,
 since $u$ and $v$ are not adjacent in $K_{\alpha+n} \setminus \frac{\alpha+n}{t}T$.
 Similarly,  the neighbor of $u$ in $T$ must be in the same partite set of size $\alpha_j$.
So all the vertices of $V_i$ are in the same partite set of size $\alpha_j$.
As $V_1\cup V_2\cup \ldots \cup V_{(\alpha+n)/t}$ is the set of all vertices of $K_{\alpha_1,\ldots,\alpha_p,n}$,
we have $t\mid\alpha_i$ for any $1\le i\le p$, a contradiction.
So no blue $K_{\alpha_1,\ldots,\alpha_p,n}$ occurs.
	
The red graph is
$$\frac{\alpha+n}{t}T+K_{k-1}(\alpha+n-1),$$
which contains $G$. Then $G$ is a subgraph of $mT+K_{k-1}(m)$.
\end{proof}

\section{Necessity: Stability and regularity method}

In this section, we prove the necessity of Theorem \ref{main} as follows.

\begin{proof}
First suppose that $p\ge 2$.
Let $T$ be any tree with $v(T)=t_1$, where $t_1=\min\{\mathrm{snd}(\alpha_1),\ldots,\allowbreak\mathrm{snd}(\alpha_p)\}$.
We shall show that $r(G, K_{\alpha_1,\ldots,\alpha_p,n}) \leq k(\alpha+n-1)+1,$
where $\alpha=\sum_{i=1}^{p}\alpha_i$. If $k=1$, then $G$ is a subgraph of $T$ as $s(G)=1$. 
If $t_1=2$, then $G=K_2$. It is trivial that $r(K_2,K_{\alpha_1,\ldots,\alpha_p,n})=\alpha+n$.
If $t_1\ge 3$, then $\alpha_i$ is even for any $1\le i\le p$.
Take the joint graph of all $T$, and then $G$ is a subgraph of $K_{1,2}$.
We shall prove that $r(K_{1,2},K_{\alpha_1,\ldots,\alpha_p,n})=\alpha+n$ for all even $\alpha_i$.
For any red-blue edge coloring of $K_{\alpha + n}$, if there is no red $K_{1,2}$, then the red graph is contained in a matching $\lfloor(\alpha + n)/{2} \rfloor K_2$.
As $\alpha_i$ is even, the complement of $K_{\alpha_1,\ldots,\alpha_p,n}$, that is $K_{\alpha_1}\cup\ldots\cup K_{\alpha_p}\cup K_n$, contains a matching $\lfloor(\alpha + n)/{2} \rfloor K_2$.
So there is a blue $K_{\alpha_1,\ldots,\alpha_p,n}$.

Now we assume that  $k \geq 2$. Denote by  $N = k(\alpha+n-1) + 1$.
Suppose that there is a red-blue edge coloring of $K_{N}$ such that it contains neither a red $G$  nor a blue $K_{\alpha_1,\ldots,\alpha_p,n}$, then it would yield a contradiction.
For this edge coloring of  $K_{N}$, let $R$ and $B$ be the spanning subgraphs with all red edges and all blue edges, respectively.
Denote by $d_{R}(v)$ and $d_{B}(v)$ the red degree and the blue degree of vertex $v$, respectively.
	
Let  $\zeta$  and  $\xi$  be small real numbers such that  $\zeta \gg \xi > 0$. Let  $\delta = \delta(\xi) > 0$  be as in Lemma \ref{wending}, and  $c$ the constant claimed in Lemma \ref{jishuqian}. Choose  $\delta_0 = \min \left\{\frac{1}{2k}, \frac{\delta}{4}\right\}$  and denote by
\begin{equation}{\label{e1}}
	d = \min  \left\{\left(\frac {\delta_ {0}}{2}\right) ^ {\alpha + 1} \left(\frac {\alpha}{c} + 2 \alpha + 1 + 2 k\right) ^ {- 1}, \frac {k \delta_ {0}}{1 + k \delta_ {0}} \left(\frac {\alpha}{c} + 2 \alpha + 1\right) ^ {- 1} \right\}
\end{equation}
and
\begin{equation*}{\label{e2}}
	\epsilon = \min  \left\{\delta_ {0}, \frac {d ^ {\Delta}}{2 (\Delta + 1)} \right\},
\end{equation*}
in which  $\Delta = \Delta (G)$.
Note that  $0 < 2\epsilon < d < \delta_0 < 1$, and then
\begin{equation}{\label{e3}}
	(d - \epsilon) ^ {\Delta} \geq d ^ {\Delta} - \Delta \epsilon d ^ {\Delta - 1} > d ^ {\Delta} - \Delta \epsilon \ge 2 (\Delta + 1) \epsilon - \Delta \epsilon = (\Delta + 2) \epsilon.
\end{equation}

We shall separate the proof into several claims as follows.
\medskip

\begin{claim}{\label{claim1}} $e(R) \geq \left( \frac{k - 1}{2k} - \delta \right) N^2$.
\end{claim}

\begin{proof}
We divide the proof into three steps.
	
\noindent
\emph{Step 1.} (Regularity partition)
Apply Lemma \ref{zhengzeyinli} to obtain a partition of  $V(R)$  as  $V_0, V_1, \ldots, V_t$  such that  $|V_1| = \dots = |V_t|$  and  $|V_0| < \epsilon N$, all but at most  $\epsilon t^2$  pairs  $(V_i, V_j)$  are  $\epsilon$-regular with  $i \neq j$. Assume  $|V_i| > r(G, K_{\alpha})$  and  $t > 1/\epsilon$. Define four disjoint graphs  $H_{irr}, H_{blue}, H_{mid}, H_{red}$  on the vertex set  $[t] = \{1, 2, \ldots, t\}$  such that for  $i \neq j$,
	\begin{itemize}
		\item $(i,j)\in E(H_{irr})$  if and only if  $(V_{i},V_{j})$  is not  $\epsilon$-regular;
		\item $(i,j)\in E(H_{blue})$  if and only if  $(V_{i},V_{j})$  is  $\epsilon$-regular and  $d_R(V_i,V_j)\leq d$;
		\item $(i,j)\in E(H_{mid})$  if and only if  $(V_{i},V_{j})$  is  $\epsilon$-regular and  $d < d_R(V_i,V_j)\leq 1 - \delta_0$;
		\item $(i,j)\in E(H_{red})$  if and only if  $(V_{i},V_{j})$  is  $\epsilon$-regular and  $d_R(V_i,V_j) > 1 - \delta_0$.
	\end{itemize}
As $d > 2\epsilon$, $t > 1/\epsilon$ and
		$e (H _ {irr}) + e (H _ {blue}) + e (H _ {mid}) + e (H _ {red}) = \binom{t}{2},$
		 we have
		\begin{align}{\label{e4}}
			e\left(H _ {blue}\right) + e\left(H _ {mid}\right) + e\left(H _ {red}\right) & \geq \binom{t}{2} - \epsilon t ^ {2} = \frac {t ^ {2}}{2} - \frac {t}{2} - \epsilon t ^ {2} \notag\\
			& \geq \frac {t ^ {2}}{2} - \epsilon t ^ {2} - \epsilon t ^ {2} > \left(\frac {1}{2} - d\right) t ^ {2}.
		\end{align}
By (\ref{e3}) and Lemma \ref{l3}, the graph induced by  $E(H_{mid}) \cup E(H_{red})$  is  $G$-free. Thus by Lemma \ref{l4},
		$$
		e (H _ {mid}) + e (H _ {red}) \leq \frac {(k - 1) t ^ {2}}{2 k},
		$$
		which, together with (\ref{e4}), implies that
\begin{equation}\label{c1}
 e\left(H _ {blue}\right) > \left(\frac {1}{2 k} - d\right) t ^ {2}.
\end{equation}        		
\emph{Step 2.} (Counting $K_{\alpha_1,\ldots,\alpha_p,1}$)
We will count the blue  $K_{\alpha_1,\ldots,\alpha_p,1}$ with the $\alpha$  vertices of $K_{\alpha_1,\ldots,\alpha_p}$ in some  $V_{i}$  and one outside  $V_{i}$.
For  $i \in [t]$, denote by  $q$  the number of the blue  $K_{\alpha_1,\ldots,\alpha_p}$ in $V_{i}$.
By Lemma \ref{jishuqian} and $K_{\alpha_1,\ldots,\alpha_p}\subseteq K_{\alpha}$,  we have $q \geq c|V_{i}|^{\alpha}$.
Set $L = N_{H_{blue}}(i)$ to be the neighborhood of vertex $i$ in $H_{blue}$. Using Lemma \ref{l1} for  $A = V_i$, $B_j = V_j$  with  $j \in L$, set
$$
H_1 = B \left[ A \cup \left(\cup_ {j \in L} B _ {j}\right) \right]
$$
as the blue graph induced by  $A \cup (\cup_{j \in L} B_j)$. For each  $j \in L$, the pair  $(V_i, V_j)$  is  $\epsilon$-regular and
$$
e_{H_1} (V _ {i}, V _ {j}) \geq (1 - d) | V _ {i} | | V _ {j} |,
$$
so there are at least
$$
d _ {H _ {blue}} (i) | V _ {i} | (q - \epsilon \alpha | V _ {i} | ^ {\alpha}) (1 - d - \epsilon) ^ {\alpha}
$$
 $K_{\alpha_1,\ldots,\alpha_p,1}$ in $H_1$ that has exactly  $\alpha$ vertices of  $K_{\alpha_1,\ldots,\alpha_p}$ in  $V_{i}$  and one vertex in  $\cup_{j\in L}B_{j}$.

We then set  $M = N_{H_{mid}}(i)$, and use Lemma \ref{l1} for  $A = V_i$,  $B_j = V_j$ with $j \in M$, and
$$
H_2 = B [ A \cup (\cup_ {j \in M} B _ {j}) ],
$$
which is the blue graph induced by  $A \cup (\cup_{j \in M} B_j)$. For each  $j \in M$,  $(V_i, V_j)$  is  $\epsilon$-regular and
$$
e_{H_2} (V _ {i}, V _ {j}) \geq \delta_ {0} | V _ {i} | | V _ {j} |,
$$
so there are at least
$$
d _ {H _ {mid}} (i) | V _ {i} | (q - \epsilon \alpha | V _ {i} | ^ {\alpha}) (\delta_ {0} - \epsilon) ^ {\alpha}
$$
$K_{\alpha_1,\ldots,\alpha_p,1}$ in  $H_2$ that has exactly  $\alpha$ vertices of  $K_{\alpha_1,\ldots,\alpha_p}$ in  $V_{i}$  and one vertex in  $\cup_{j\in M}B_{j}$. Note that
$$
\left(\cup_ {j \in L} B _ {j}\right) \cap \left(\cup_ {j \in M} B _ {j}\right) =\emptyset,
$$
so there are at least
$$
d _ {H _ {blue}} (i) | V _ {i} | (q - \epsilon \alpha | V _ {i} | ^ {\alpha}) (1 - d - \epsilon) ^ {\alpha} + d _ {H _ {mid}} (i) | V _ {i} | (q - \epsilon \alpha | V _ {i} | ^ {\alpha}) (\delta_ {0} - \epsilon) ^ {\alpha}
$$
blue $K_{\alpha_1,\ldots,\alpha_p,1}$ in  $B$  having exactly $\alpha$ vertices of  $K_{\alpha_1,\ldots,\alpha_p}$ in  $V_{i}$  and one vertex outside  $V_{i}$.

\emph{Step 3.} (Contradiction via $n_0$)
Let  $n_0$ be the largest $n$  such that   $K_{\alpha_1,\ldots,\alpha_p,n}$  is contained in $B$.
Averaging over all the $q$  blue  $K_{\alpha_1,\ldots,\alpha_p}$ in  $V_{i}$, we have
\begin{align}
	n_0 \geq& | V _ {i} | \left(1 - \frac {\epsilon \alpha}{c}\right) \left(d _ {H _ {blue}} (i) (1 - d - \epsilon) ^ {\alpha} + d _ {H _ {mid}} (i) (\delta_ {0} - \epsilon) ^ {\alpha}\right) \notag\\
	\geq& N \big (\frac {1 - \epsilon}{t} \big) \big (1 - \frac {\epsilon \alpha}{c} \big) (d _ {H _ {blue}} (i) (1 - d - \epsilon) ^ {\alpha} + d _ {H _ {mid}} (i) (\delta_ {0} - \epsilon) ^ {\alpha}).\notag
\end{align}
Therefore we obtain
\begin{align} {\label{e6}}
	\frac {n_0}{N} &\geq (1 - \epsilon) \big(1 - \frac {\epsilon \alpha}{c} \big) \Big(\frac {2 e (H _ {blue})}{t ^ {2}} (1 - d - \epsilon) ^ {\alpha} + \frac {2 e (H _ {mid})}{t ^ {2}} (\delta_ {0} - \epsilon) ^ {\alpha} \Big) \notag\\
	&> \left(1 - (\frac {\alpha}{c} + 1) \epsilon\right) \left(\frac {2 e (H _ {blue})}{t ^ {2}} (1 - \alpha (d + \epsilon)) + \frac {2 e (H _ {mid})}{t ^ {2}} (\delta_ {0} - \epsilon) ^ {\alpha}\right) \notag\\
	&> \left(1 - (\frac {\alpha}{c} + 1) d\right) \left(\frac {2 e \left(H _ {blue}\right)}{t ^ {2}} (1 - 2 \alpha d) + \frac {2 e \left(H _ {mid}\right)}{t ^ {2}} \left(\frac {\delta_ {0}}{2}\right) ^ {\alpha}\right)  \notag\\
	&> \left(1 - (\frac {\alpha}{c} + 2 \alpha + 1) d\right) \frac {2 e (H _ {blue})}{t ^ {2}} + \left(1 - (\frac {\alpha}{c} + 1) d\right) \left(\frac {\delta_ {0}}{2}\right)^{\alpha} \frac {2 e (H _ {mid})}{t ^ {2}}.
\end{align}
Recall that it is enough to show that  $n_0 > N/k$, so for a contradiction, assume that
$	n_0 \leq N/k$.
Ignoring the term  $e(H_{mid})$  in (\ref{e6}), from $n_0 \leq N/k$ and (\ref{e1}), we obtain
	\begin{align}
		e (H _ {blue}) &< \left(1 - \left(\frac {\alpha}{c} + 2 \alpha + 1\right) d\right) ^ {- 1} \frac {t ^ {2}}{2 k} \notag\\
		&\leq \left(1 - \frac {k \delta_ {0}}{1 + k \delta_ {0}}\right) ^ {- 1} \frac {t ^ {2}}{2 k} = \left(\frac {1}{2 k} + \frac {\delta_ {0}}{2}\right) t ^ {2}.\label{c3}
	\end{align}
	Furthermore, observe that (\ref{e1}) implies that
	$$
	\left(\frac {\alpha}{c} + 1\right) d <   \left(\frac {\alpha}{c} + 2 \alpha + 1\right) d \leq \frac {k \delta_ {0}}{1 + k \delta_ {0}} \leq k \delta_ {0} \leq   \frac {1}{2},
	$$
	and consequently,
	$$
	1 - \left(\frac {\alpha}{c} + 1\right) d > \frac {1}{2}.
	$$
	Hence, by (\ref{c1}), (\ref{e6}) and $n_0\le N/k$, it follows
	\begin{align}
		\frac {e (H _ {mid})}{2} \left(\frac {\delta_ {0}}{2} \right) ^ {\alpha} &<   e (H _ {mid}) \left(\frac {\delta_ {0}}{2} \right) ^ {\alpha} \left(1 - \left(\frac {\alpha}{c} + 1 \right) d \right) \notag\\
		&\leq \frac {n_0 t ^ {2}}{2 N} - \left(1 - \left(\frac {\alpha}{c} + 2 \alpha + 1\right) d\right) e (H _ {blue}) \notag\\
		&<   \left(\frac {1}{2 k} - \left(1 - \left(\frac {\alpha}{c} + 2 \alpha + 1\right) d\right) \left(\frac {1}{2 k} - d\right)\right) t ^ {2} \notag\\
		&= \left(1 + \left(\frac {\alpha}{c} + 2 \alpha + 1\right) \left(\frac {1}{2 k} - d\right)\right) d t ^ {2} \notag\\
		&<   \frac {1}{2 k} \left(\frac {\alpha}{c} + 2 \alpha + 1 + 2 k\right) d t ^ {2} <   \left(\frac {\delta_ {0}}{2}\right) ^ {\alpha + 1} t ^ {2}.\notag
	\end{align}
Therefore, we have
\begin{equation}\label{c4}
e \left(H _ {mid}\right)<\delta_ {0}t^{2}.
\end{equation}

By (\ref{e4}),  (\ref{c3}) and (\ref{c4}),
$$
e (H _ {red}) > \left(\frac {1}{2} - d\right) t ^ {2} - \left(\frac {1}{2 k} + \frac {\delta_ {0}}{2}\right) t ^ {2} - \delta_ {0} t ^ {2} \geq \left(\frac {k - 1}{2 k} - \frac {5 \delta_ {0}}{2}\right) t ^ {2},
$$
and consequently, from the definition of  $H_{red}$, we obtain
\begin{align}
	e (R) &\geq e (H _ {r e d}) \left(\frac {N - \epsilon N}{t} \right) ^ {2} (1 - \delta_ {0}) > \left(\frac {k - 1}{2 k} - \frac {5 \delta_ {0}}{2} \right) (1 - 2 \epsilon) (1 - \delta_ {0}) N ^ {2} \notag\\
	&= \frac {k - 1}{2 k} \left(1 - \frac {5 k \delta_ {0}}{k - 1} \right) (1 - 2 \epsilon) (1 - \delta_ {0}) N ^ {2} \notag\\
	&> \frac {k - 1}{2 k} \left(1 - \left(\frac {5 k}{k - 1} + 3\right) \delta_ {0}\right) N ^ {2} > \left(\frac {k - 1}{2 k} -4\delta_ {0}\right) N ^ {2} \geq \left(\frac {k - 1}{2 k} - \delta\right) N ^ {2},\notag
\end{align}
which completes the proof of Claim \ref{claim1}.
\end{proof}

Claim \ref{claim1} ensures that the red graph  $R$  has the edge density for Lemma \ref{wending} to apply for any  $\xi > 0$. With  $G$  the forbidden graph, Lemma \ref{wending} gives a partition  $C_1, \ldots, C_k$  of  $K_N$  that satisfies
	\begin{itemize}
		\item $\frac{N}{k} - \xi N < |C_i| < \frac{N}{k} + \xi N $ $(1 \leq i \leq k)$;
		\item all but at most  $\xi N^2$  pairs  $\{x,y\}$  with  $x\in C_i$, $y\in C_j$ $(i\neq j)$ belong to  $E(R)$;
		\item at most  $\xi N^2$  pairs  $\{x,y\}$  with  $x,y$  in the same  $C_i$  belong to  $E(R)$;
		\item each  $x \in C_i$  has  $d_{C_i}(x) \leq d_{C_j}(x)$ $(i \neq j)$.
	\end{itemize}
For each  $C_i$, we define a subset  $C_i^{\prime}$  of  $C_i$  with
$$
C _ {i} ^ {\prime} = \{x \in C _ {i} \mid d _ {R} (x, C _ {j}) \geq (1 - 2 \sqrt {\xi}) | C _ {j} |, \forall j \neq i \},
$$
where  $d_R(x, C_j)$  is the number of red neighbors of  $x$  in  $C_j$, namely,  $|N_R(x) \cap C_j|$.
We have the following claim.

\begin{claim} \label{claim4}
$|C_i^{\prime}| \geq (1 - k^2\sqrt{\xi})|C_i|$  for  $i = 1, 2, \dots, k$.
\end{claim}
\begin{proof}
	Suppose to the contrary that $|C_i^{\prime}| < (1 - k^2\sqrt{\xi})|C_i|$  for some  $C_i$, and then
	$$
	\left| C _ {i} \backslash C _ {i} ^ {\prime} \right| > k ^ {2} \sqrt {\xi} \left| C _ {i} \right| \geq k ^ {2} \sqrt {\xi} (1 / k - \xi) N.
	$$
	However, each vertex  $x \in C_i \backslash C_i^{\prime}$  has
	$$
	d _ {B} (x, C _ {j}) \geq 2 \sqrt {\xi} | C _ {j} | \geq 2 \sqrt {\xi} (1 / k - \xi) N.
	$$
	So for small  $\xi$  and large enough  $n$, we have
	$$
	e _ {B} (C _ {i}, C _ {j}) \geq k ^ {2} \sqrt {\xi} (1 / k - \xi) N \cdot 2 \sqrt {\xi} (1 / k - \xi) N = 2 k ^ {2} \xi (1 / k - \xi) ^ {2} N ^ {2} > \xi N ^ {2},
	$$
	which yields a contradiction.
\end{proof}

\begin{claim} \label{claim5} For $1 \leq i \leq k$, the red subgraph  $R[C_i^{\prime}]$  induced by  $C_i^{\prime}$  contains no $mT$, where $T$ is any tree on $t_1$ vertices with $t_1=\min\{\mathrm{snd}(\alpha_1),\ldots,\allowbreak\mathrm{snd}(\alpha_p)\}$.
\end{claim}
\begin{proof}
Suppose not and let $\{u_1,u_2,\ldots,u_{mt_1}\}$ be the vertex set of the $mT$ in  $R[C_1^{\prime}]$.
Since for any  $x \in C_1^{\prime}$,
	$$
	d _ {R} \left(x, C _ {2} ^ {\prime}\right) \geq \left(1 - 2 \sqrt {\xi}\right) | C _ {2} | - k ^ {2} \sqrt {\xi} | C _ {2} | \geq \left(1 - 2 k ^ {2} \sqrt {\xi}\right) | C _ {2} |,
	$$
we have
	$$
	\left| \cap_ {i = 1} ^ {mt_1} N _ {R} \left(u _ {i}, C _ {2} ^ {\prime}\right) \right| \geq (1 - 2mt_1 k ^ {2} \sqrt {\xi}) | C _ {2} | \geq m.
	$$
	for large enough  $n$. Let us construct a red  $G$  as follows.
	We begin by choosing  $m$  vertices  $u_{mt_1+ 1},\ldots ,u_{m(t_1+1)}$  from  $\cap_{i = 1}^{mt_1}N_R(u_i,C_2^{\prime})$. So we have
	$$
	\left| \cap_ {i = 1} ^ {m(t_1+1)} N _ {R} \left(u _ {i}, C _ {3} ^ {\prime}\right) \right| \geq (1 - 2m(t_1+1) k ^ {2} \sqrt {\xi}) | C _ {3} | \geq m
	$$
	for large enough  $n$, such that we can choose  $m$  vertices from  $\cap_{i=1}^{m(t_1+1)} N_R(u_i, C_3^{\prime})$. Repeating this process, we obtain a red  $mT + K_{k-1}(m)$, and hence a red  $G$, a contradiction.
\end{proof}
For $1\le i\le k$, we partition $C_i\setminus C'_i$ into $Z_{i1}$ and $Z_{i2}$
such that
\[
Z_{i1}=\{w\in C_i\setminus C'_i: d_R(w,C'_j)\ge \zeta|C'_j|,\forall j\neq i\}.
\]

\begin{claim} \label{claim6}
Each  $w \in Z_{i1}$  has  $d_R(w, C_i^{\prime}) \leq m - 1$ for $1\le i\le k$.
\end{claim}
\begin{proof}
	We verify the claim for $C_1$. For contradiction, assume some $w \in Z_{11}$  satisfies  $d_R(w, C_1^{\prime}) \geq m$. Let  $M_1 = \{y_{11}, \dots, y_{1m}\} \subseteq N_R(w, C_1^{\prime})$.
Since $d_R(y, C_2) \geq (1 - 2\sqrt{\xi})|C_2|$ for every $y \in M_1$, one has  $|\cap_{y \in M_1} N_R(y, C_2)| \geq (1 - 2m\sqrt{\xi})|C_2|$.
As $d_R(w, C_2^{\prime}) \geq \zeta |C_2^{\prime}|$  and  $0 < \xi \ll \zeta < 1$, we have
	$$
	\left| N _ {R} \left(w, C _ {2} ^ {\prime}\right) \cap \left(\cap_ {y \in M _ {1}} N _ {R} (y, C _ {2})\right) \right| \geq \zeta \left| C _ {2} ^ {\prime} \right| + (1 - 2 m \sqrt {\xi}) | C _ {2} | - \left| C _ {2} ^ {\prime} \right| \geq m.
	$$
Then we can choose $m$	vertices from $N_{R}(w, C_{2}^{\prime}) \cap \left( \cap_{y \in M_{1}} N_{R}(y, C_{2}) \right)$. Repeating this process produces a red $K_{k+1}(1, m, \dots, m)$  that contains a red $G$, a contradiction.
\end{proof}
\begin{claim} \label{claim7} 
$|Z_{i2}| \leq (k - 1)(r(G, K_{\alpha_1,\ldots,\alpha_p}) - 1)$ for $i = 1, 2, \dots, k$.
\end{claim}
\begin{proof}
	Suppose to the contrary that $|Z_{i2}| \geq (k - 1)(r(G, K_{\alpha_1,\ldots,\alpha_p}) - 1)+ 1$  for some  $i$. Then for each $w \in Z_{i2}$, $d _ {R} \big(w, C _ {j} ^ {\prime}\big) <   \zeta | C _ {j} ^ {\prime}|$ for some  $j \neq i$.
So we extract a subset $Z_0 \subseteq Z_{i2}$  with
	$$
	| Z _ {0} | \geq \left\lceil \frac {| Z _ {i 2} |}{k - 1} \right\rceil \geq r (G, K_{\alpha_1,\ldots,\alpha_p}),
	$$
	where all $w\in Z_0$  satisfy $d _ {R} \left(w, C _ {\ell} ^ {\prime}\right) <   \zeta | C _ {\ell} ^ {\prime} |$ for fixed $\ell \neq i$. For sufficiently large $n$,
	$$
	d _ {R} \left(w, C _ {i} ^ {\prime}\right) \leq d _ {R} \left(w, C _ {i}\right) \leq d _ {R} \left(w, C _ {\ell}\right) \leq d _ {R} \left(w, C _ {\ell} ^ {\prime}\right) + \left| C _ {\ell} \backslash C _ {\ell} ^ {\prime} \right| <   2 \zeta \left| C _ {\ell} ^ {\prime} \right|.
	$$
	This implies
	\begin{align}
		d _ {B} \left(w, C _ {i} ^ {\prime} \cup C _ {\ell} ^ {\prime}\right) &\geq | C _ {i} ^ {\prime} | + | C _ {\ell} ^ {\prime} |- \left(d _ {R} \left(w, C _ {i} ^ {\prime}\right) + d _ {R} \left(w, C _ {\ell} ^ {\prime}\right)\right) \notag\\
		&\geq \left| C _ {i} ^ {\prime} \right| + \left| C _ {\ell} ^ {\prime} \right| - 3 \zeta \left| C _ {\ell} ^ {\prime} \right| \notag\\
		&> (1 - 3 \zeta) \left(\left| C _ {i} ^ {\prime} \right| + \left| C _ {\ell} ^ {\prime} \right|\right).\notag
	\end{align}
	Since  $|Z_0| \geq r(G, K_{\alpha_1,\ldots,\alpha_p})$ and no red  $G$ exists, $Z_0$ induces a blue $K_{\alpha_1,\ldots,\alpha_p}$. The $\alpha$ vertices of this $K_{\alpha_1,\ldots,\alpha_p}$ have at least
	$$
	(1 - 3 \alpha \zeta) \left(\left| C _ {i} ^ {\prime} \right| + \left| C _ {\ell} ^ {\prime} \right|\right) \geq n
	$$
	common blue neighbors in $C_i^{\prime} \cup C_\ell^{\prime}$, yielding a blue $K_{\alpha_1,\ldots,\alpha_p,n}$, a contradiction.
\end{proof}
By Claim \ref{claim7}, we have
$\left| \cup_ {i = 1} ^ {k} Z _ {i 2} \right| \leq k (k - 1)(r(G, K_{\alpha_1,\ldots,\alpha_p}) - 1)$,
and we can partition  $Z_{i2}$  into  $k - 1$  subsets as
 $Z_{i2}^{1}, Z_{i2}^{2}, \ldots, Z_{i2}^{i - 1}, Z_{i2}^{i + 1}, \ldots, Z_{i2}^{k}$  such that for any  $j \neq i$,  $w \in Z_{i2}^{j}$  if and only if  $j$  is the minimum integer with  $d_R(w, C_j^{\prime}) < \zeta | C_j^{\prime} |$. It is easy to see that the sets  $Z_{i2}^j$  are pairwise disjoint when  $i \neq j$. Note that
$$
\cup_ {i = 1} ^ {k} Z _ {i 2} = \cup_ {i = 1} ^ {k} \big (\cup_ {j: j \neq i} Z _ {i 2} ^ {j} \big) = \cup_ {j = 1} ^ {k} \big (\cup_ {i: i \neq j} Z _ {i 2} ^ {j} \big) = \cup_ {i = 1} ^ {k} \big (\cup_ {j: j \neq i} Z _ {j 2} ^ {i} \big),
$$
which implies that
$$
V = \cup_ {i = 1} ^ {k} \big (C _ {i} ^ {\prime} \cup Z _ {i 1} \cup Z _ {i 2} \big) = \cup_ {i = 1} ^ {k} \big [ C _ {i} ^ {\prime} \cup Z _ {i 1} \cup \left(\cup_ {j: j \neq i} Z _ {j 2} ^ {i}\right) \big ].
$$
This ensures us to have some  $h \in \{1, 2, \dots, k\}$  such that
$$
\left| C _ {h} ^ {\prime} \cup Z _ {h 1} \cup \big(\cup_ {j: j \neq h} Z _ {j 2} ^ {h}\big)\right| \geq \left\lceil \frac {| V |}{k} \right\rceil \geq n + \alpha.
$$
In Claim \ref{claim5}, there is no one kind of $mT$ in $R[C_h^{\prime}]$ with $v(T)=t_1$. Let $c_0=c_0(\alpha_1,\ldots,\alpha_p)$ be the number of different kinds of $T$.
After removing the vertices from all kinds of $T$ in  $C_h^{\prime}$, we get the vertex set $Q$ such that
$$
|Q| \geq \left|C_{h}^{\prime} \right|-c_0mt_1.
$$
Note that any connected graph with $t_1$ vertices contains a spanning tree.
So the order of the largest component in $R[Q]$ is at most $t_1-1$.
With a similar proof to that in Claim \ref{claim6}, for any vertex  $v \in C_h^{\prime}$, we have  $d_R(v, C_h^{\prime}) \leq m - 1$.
Now we can find a subset  $Q^{\prime}$  of $Q$ such that each vertex in $Q^{\prime}$ is adjacent to any vertex in  $Z_{h1} \cup (\cup_{j:j \neq h} Z_{j2}^h) \cup (C_h^{\prime} \backslash Q)$  by blue edges. Then we have
\begin{align}
	|Q^{\prime}|
	\geq& |Q|- (t_1-1)(m|C_{h}^{\prime} \backslash Q|+m|Z_{h 1}|+\big(\cup_{j:j \neq h} Z_{j2}^{h}\big) \zeta N) \notag\\
	\geq& |C_{h}^{\prime}|-c_0mt_1- t_1(c_0 m^2 t_1+mk^{2} \sqrt{\xi} |C_{h}|+(k-1)^2(r(G, K_{\alpha_1,\ldots,\alpha_p}) - 1) \zeta N)  \notag\\
	\geq& (1-k^{2} \sqrt{\xi}) |C_{h}|-t_1mk^{2} \sqrt {\xi} |C_{h}|-t_1(k-1)^2(r(G, K_{\alpha_1,\ldots,\alpha_p}) - 1) \zeta N-c_0mt_1(mt_1+1)\notag\\
	\geq& \big(1-k^{2} \sqrt{\xi}-t_1mk^{2} \sqrt{\xi}\big) \big(\frac {1}{k}-\xi\big) N-t_1(k-1)^2(r(G, K_{\alpha_1,\ldots,\alpha_p}) - 1) \zeta N-c_0mt_1(mt_1+1) \notag\\
	\geq& t_1\alpha,\notag
\end{align}
where $\zeta$ is small and $n$ is large enough. Therefore, we get a red graph with $t_1\alpha$ vertices such that the order of the largest component is at most $t_1-1$.
By the pigeonhole principle, there are $n_1$ connected graphs with each of $q_1$ vertices such that $n_1 q_1\geq \alpha$. Since $q_1\le t_1-1$ and $q_1\mid \alpha_i$ for any $1\le i\le p$, we can embed the $\alpha/q_1$ connected graphs with each of $q_1$ vertices in the $K_{\alpha_1,\ldots,\alpha_p}$.
Note that the $\alpha$ vertices of the $\alpha/q_1$ connected graphs are connected with all the other vertices in $C _ {h} ^ {\prime} \cup Z _ {h 1} \cup \big(\cup_ {j: j \neq h} Z _ {j 2} ^ {h}\big)$ by blue edges completely.
Therefore, we have a blue $K_{\alpha_1,\ldots,\alpha_p,n}$, a contradiction.

If $p=1$,  since $G$ is a subgraph of $ mT + K_{k-1}(m)$ for every tree $T$ with $v(T)=\mathrm{snd}(\alpha_1)$, then $K_{\alpha_1,\alpha_1,n-\alpha_1}$ are $G$-good for large $n$. 
As $K_{\alpha_1,n}$ is a subgraph of  $K_{\alpha_1,\alpha_1,n-\alpha_1}$, we have $K_{\alpha_1,n}$ are $G$-good for large $n$.
\end{proof}

\end{spacing}

\end{document}